\documentclass{amsart}
\usepackage{amsmath,amsthm,amsfonts,amssymb,latexsym,mathrsfs,graphicx}

\usepackage{hyperref}

\usepackage{enumerate}
\usepackage[shortlabels]{enumitem}
\usepackage{color}

\newtheorem{theorem}{Theorem}[section]

\newtheorem{lemma}[theorem]{Lemma}
\newtheorem{remark}[theorem]{Remark}

\theoremstyle{definition}

\newcommand{\Irr}{{\mathrm {Irr}}}
\newcommand{\cd}{{\mathrm {cd}}}
\newcommand{\cod}{{\mathrm {cod}}}
\newcommand{\Aut}{{\mathrm {Aut}}}

\newcommand{\kernel}{{\mathrm {ker}}}
\newcommand{\PSL}{{\mathrm {PSL}}}

\newcommand{\SL}{{\mathrm {SL}}}

\newcommand{\Syl}{\mathrm{Syl}}

\newcommand{\orb}{\mathrm{orb}}

\begin{document}

\title{On the multiplicities of the character codegrees of finite groups}

\author[Z. Akhlaghi et al.]{Zeinab Akhlaghi}
\address{Zeinab Akhlaghi, Faculty of Math. and Computer Sci., \newline Amirkabir University of Technology (Tehran Polytechnic), 15914 Tehran, Iran.\newline
School of Mathematics,
Institute for Research in Fundamental Science(IPM)
P.O. Box:19395-5746, Tehran, Iran.}
\email{z\_akhlaghi@aut.ac.ir}

\author[]{Mehdi Ebrahimi}
\address{Mehdi Ebrahimi,  School of Mathematics,
	Institute for Research in Fundamental Sciences (IPM),
	P. O. Box 19395-5746, Tehran, Iran.}
\email{m.ebrahimi.math@ipm.ir}

\author[]{Maryam Khatami}
\address{Maryam Khatami, Department of Pure Mathematics, 
	Faculty of Mathematics and Statistics, 
	University of Isfahan,
	Isfahan, 81746-73441, Iran.
	}
\email{m.khatami@sci.ui.ac.ir}

\thanks{
The first author  is supported by a grant from IPM (No. 1400200028). The  second  author  is   supported by a  grant from the School of Mathematics, Institute for Research in
Fundamental Sciences (IPM).}


\begin{abstract} 
	Let  $ G $   be a finite group and  $ \chi $ be  an irreducible character of      $ G $, the number      $ \cod(\chi) = |G: \kernel(\chi)|/\chi(1) $           is called  the codegree of     $ \chi $. Also,   $ \cod(G) = \{ \cod(\chi) \  | \ \chi \in \Irr(G) \} $. For $d\in\cod(G)$, the multiplicity of $d$ in $G$, denoted by $m'_G(d)$, is the number of irreducible characters of $G$ having codegree $d$. A finite group $G$ is called a  $T'_k$-group for some integer $k\geq 1$, if there exists $d_0\in\cod(G)$ such that $m'_G(d_0)=k$ and for every $d\in\cod(G)-\{d_0\}$, we have $m'_G(d)=1$. In this note we characterize finite $T'_k$-groups completely, where $k\geq 1$ is an integer.    
\end{abstract}

\maketitle
\section{Introduction}
$\noindent$ Let $G$ be a finite group. Also let ${\rm cd}(G)$ be the set of all irreducible character degrees of $G$, that is,
${\rm cd}(G)=\{\chi(1)|\;\chi \in \Irr(G)\} $, where ${\rm Irr}(G)$ is the set of all complex irreducible characters of $G$.
For a character $\chi$ of $G$, the codegree of $\chi$ is defined as ${\rm cod}(\chi):=|G:{\rm ker}(\chi)|/\chi(1)$  (see \cite{Qian}).
Set ${\rm cod}(G):=\{{\rm cod}(\chi)|\; \chi \in {\rm Irr}(G)\}$. It is well known that ${\rm cod}(G)$ may be used to provide information on the structure of the group $G$ (see for instance \cite{DL},  \cite{group} and \cite{Qian}).

For $d\in {\rm cd}(G)$, the multiplicity of degree $d$, denoted by $m_G(d)$, is the number of irreducible characters of $G$ having the same degree $d$, i.e., $m_G(d)=|\{\chi \in {\rm Irr}(G)|\chi(1)=d\}|$.  Many people considered the relations between the multiplicities of irreducible character degrees of finite groups and the structure of the groups. Let ${\rm Irr_1}(G)$ be the set of nonlinear irreducible characters in ${\rm Irr}(G)$ and ${\rm cd}_1(G)$ be the set of degrees of the characters in ${\rm Irr}_1(G)$. If for a finite group $G$, there exists a nonnegative integer $n$ such that $|{\rm cd}_1(G)|=|{\rm Irr}_1(G)|-n$, then $G$ is called a $D_n$-group. In \cite{D_0}, Berkovich et al. gave the classification of $D_0$-groups. Then, Berkovich and Kazarin in \cite{d1} and \cite{berkovich} classified all $D_1$-groups. Also $D_2$-groups have been classified by Liu and Lu (see \cite{K} and \cite{d2}).

Inspired by the definition of $D_n$-groups,  we define new families of finite groups called $D^\prime_n$-groups, that are the groups $G$ satisfying  $|{\rm cod}(G)|=|{\rm Irr}(G)|-n$, for some nonnegative integer $n$.  The second author classified $D'_0$-groups, and showed that they are $C_2$ and $S_3$ (see \cite{ebrahimi}).

A finite group $G$ is said to be a $T_k$-group for some integer $k\geq 1$, if there exists a nontrivial degree
$d_0 \in {\rm cd}(G)$ such that $m_G(d_0)=k$ and that for every $d\in {\rm cd}(G)-\{1,d_0\}$, the multiplicity of $d$
in $G$ is trivial, that is, $m_G(d)=1$. Tong-Viet in \cite{hung} showed that if $G$
is a non-solvable $T_k$-group for some integer $k\geq 1$, then $k=2$ and $G$ is isomorphic to
${\rm PSL}_2(5)$ or ${\rm PSL}_2(7)$. Analogously, one can define a $T'_k$-group for some integer $k\geq 1$ based on the set ${\rm cod}(G)$, that is a finite group $G$ in which there exists
$d_0 \in {\rm cod}(G)$ such that $m'_G(d_0)=k$
and for every $d\in{\rm cod}(G)$$-\{d_0\}$, we have
$m'_G(d)=1$. In this paper we wish to study $T'_k$-groups, for some integer $k\geq 1$. In fact, we characterize these groups completely and we prove the following theorem:
\bigskip

\textbf{Main Theorem.}\textit{ Let $G$ be a finite $T_k'$-group, for some integer $k\geq 1$. Then one of the following occurs:
	\begin{itemize}
		\item[(a)]  $G$ is non-solvable if and only if  $k=2$ and $G\cong \PSL_2(5)$ or $G\cong \PSL_2(7)$.
		\item[(b)] $G$ is solvable if and only if one of the  following holds:
		\begin{itemize}
			\item[(1)] $G\cong C_4$; $S_4$; Suz$(2)\cong C_5\rtimes C_4$ or Dic$_{12}\cong$SmallGroup(12,1). Moreover $k=2$.
			\item[(2)] $G\cong C_2\times S_3$; $D_{18}$ or Frobenius group $C_3^2\rtimes Q_8$. Moreover $k=3$.
			\item[(3)] $G$ is an elementary abelian $p$-group, for some prime $p$, or $G$ is an extra-special $2$-group. In particular $k=|G/G'|-1$.
			\item[(4)] $G$ is isomorphic to the Frobenius group $C_2^{\beta}\rtimes C_{2^{\beta}-1}$, for some integer $\beta$,  where $p=2^{\beta}-1$ is prime. Moreover $k=p-1$.
			\item[(5)] $G$ is isomorphic to the Frobenius group $C_p^{\beta}\rtimes C_2$, for some prime $p$ and  integer $\beta$.  Moreover $k=(p^{\beta}-1)/2$.
		\end{itemize}
	\end{itemize}
}

If $N\unlhd G$ and $\theta\in {\rm Irr}(N)$, then  ${I}_G(\theta)$ denotes the inertia group of $\theta$ in $G$, and ${\rm Irr}(G|\theta)$ denotes the set of all irreducible constituents of $\theta^G$. Moreover, $\Irr(G|N)=\Irr(G)-\Irr(G/N)$. By a rational group, we mean a group whose all irreducible characters are rational-valued. By $n_p$ we  mean, the $p$-part of $n$, where $n$ is a natural number while, by  $n_q(G)$ we mean  the number of Sylow $q$-subgroups of $G$. For the rest of the notations, we follow \cite{isaac}.

\section{ Preliminaries}
In this section, we recall some known results of irreducible characters which will be needed throughout the proofs.

%\begin{lemma}\label{size Nq} \rm(\cite[Proposition 8]{carlo}\rm)
%	If $(G,M)$ satisfies $N_q$, then $(|M|-1)/(|{\bf C}_M(Q)|-1)=n_q(G)$, where $n_q(G)$ is the number of Sylow $q$-subgroups of $G$ and $Q\in {\rm Syl}_q(G)$.
%\end{lemma}
\begin{remark}\label{orbit}
	Let $K$ be a normal abelian subgroup of $G$. Then $|G:I_G(\lambda)|\leq |K|-1$, for each $\lambda\in \Irr(K)$.
\end{remark}
The above remark follows by the fact that there is a correspondence between   $G$-orbits of $\Irr(K)$ and $G$-conjugacy classes of $K$. We use this remark, several times through the proofs.

\begin{lemma}	\label{size Nq}(see \cite[Lemma 2.2]{zzk})
	Let   $ G $     be a finite group acting on a module   $ M $  over a finite field and $ 1\in S \subseteq M $ be a union of some   $ G $-orbits of   $ M $. Also, assume   $ q $ is  a prime divisor of the order of $ G $
	such that for every              $ v \in M \setminus S$,    $ C_{G}(v) $  contains a Sylow    $ q $-subgroup of   $ G $  as a normal subgroup.
	Assume
	%     orb$_{G}(x) \neq $orb$_{G}(y) $
	% for each      $ x, y \in S $
	%	and also
	for each         $ 1\not =x \in S $,        $ q  $ divides  $ |\orb_{G}(x)| $. Then,    $ |M| - |S| = n_{q}(G)(|C_{M}(Q)|-1) $  %where     $ p^{\beta} = |C_{M}(Q)| $
	for some     $ Q \in \Syl_{q}(G) $.
	%and integer  $ \beta $.
\end{lemma}	

%\begin{defn}(See \cite{carlo})
% Let $ G $ be a finite group acting on a module $ M $ over a finite field, and $ q $ be a prime divisor of the order of $ G/C_{G}(M) $. We write that
%the pair  $ (G, M) $ satisfies $ N_{q} $ if for every $ v \in M\backslash\{0\} $,  $ C_{G}(v)  $ contains a Sylow $ q $-subgroup
%of $ G $ as a normal subgroup.
%\end{defn}

%\begin{lemma}(See \cite[Lemma 3.3]{acdel}) \label{acdel 3.3}
%	Let $ H = \PSL_{2}(u^{\alpha}) $. Then, for any odd prime $ q $, there does not exist any $ H $-module
%	$ M $ such that $ (H, M) $ satisfies $ N_{q}  $.

%\end{lemma}

%\begin{lemma}\label{khatoon}(See \cite[Proposition 8]{carlo})
%Let $ G $ be a finite group acting on a module $ M $ over a finite %field, and $ q $ be a prime divisor of the order of $ G/C_{G}(M) $.
%	If  $ (G, M) $ satisfies $ N_{q} $, then $(|M|-1)/(|C_M(Q)|-1)=n_q(G)$, where  $n_q(G)$ is the  number of Sylow $q$-subgroups of $G$ and $Q\in \Syl_q(G)$.
%\end{lemma}	

%\begin{lemma}(see \cite[Theorem 6.15]{issac}) \label{6.15}
% Let $ A \trianglelefteq G $ be abelian. Then $ \chi(1) $ dividies $ |G:A| $
%for all $\chi \in  \Irr(G) $.
%\end{lemma}
%\begin{lemma} (see \cite[Corollary 2.3]{1}) \label{p-group}
%If  $ G $   is a nontrivial    $ p $-group for some prime       $ p $, then        $ p \in \cod(G) $.
%\end{lemma}

%The following results are taken  from  \cite[Theorems 3 and 4]{6},  \cite[Theorem 6.4.7]{7}, \cite[Lemmas  2.4 and 3.2]{d2}, \cite[Lemma 2.3]{hung}   and \cite{atlas, kerber, Schmid}.

\begin{lemma}  \label{ somelemmas}(see \cite[Lemma 2.3]{hung})
	Let $S$ be a finite non-abelian simple group.
	
	(1) If $S$ is a sporadic simple group, the Tits group or an alternating group of degree at least 7, then $S$ has two nontrivial irreducible characters, with distinct degrees and both extendible to $\Aut(S)$.
	
	(2) If $S$ is a simple group of Lie type in characteristic $p$ and $S\not \cong$ $^2F_4(2)'$, then the Steinberg character of $ S $, denoted by $St_S$, of degree $|S|_p$, is extendible to $\Aut(S)$. Furthermore, if $S\not\cong \PSL_2(3^f)$ for some integer $f\geq 2$, then $S$ possesses an irreducible character $\theta$ such that $\theta(1)\not=|S|_p$ and $\theta$ also extends to $\Aut(S)$.
	
	%$(a)$ If   $S$ is a  simple group of Lie type over a field of characteristic $ p $ and St denote the Steinberg character of $ S $. Then, St is extendible to $ \Aut(S) $ (the automorphism group of $ S $) and its degree is $ |S|_{p} $.
	
	%$(b)$ If $S$ is either  a non-abelian simple group of Lie type, except for $A^{\pm }_{2}(q^{2})$ and $A_1(q)$,  a sporadic simple group or an  alternating group $A_n$ for $n\geq 7$, then there are three nonlinear irreducible characters of $S$  of different degrees  which are extendible to $ \Aut(S) $.
	
	%$(c)$ If $ S\cong A^{\pm }_{2}(q^{2})$,   then there are two nonlinear irreducible characters with different degrees  which are extendible to $  \Aut(S) $. Furthermore, if $ S  \cong A_{1}(q)$, where $ q = p^{f} $  and  $p\not =3$,  then $ S $ possesses an irreducible character $ \theta $ such that $ \theta(1)  = q + \delta
	%\neq |S|_{p} $,   where $ q \equiv \delta\ (\textrm{mod}\ 3) $, $ \delta \in \{ \pm 1\} $ and   $ \theta $  extends to $ \Aut(S) $.
\end{lemma}
\begin{lemma}(see \cite[ Lemma 5]{ble}) \label{d2}
	Let $ N $ be a minimal normal subgroup of $ G $ such that $ N =
	S_{1} \times  \dots \times S_{t} $, 	 where $ S_i \cong S $  is a non-abelian simple group. If $ \sigma \in \Irr(S) $ extends to $ \Aut(S) $, then $ \sigma \times \dots \times \sigma \in \Irr(N) $ extends to $ G $.
	%an irreducible constituent of 	$ \chi_{N}  $, and suppose $\theta_{1}, \cdots, \theta_{t} $  are the distinct conjugates of $ \theta $ in $ G $. Then $ \chi_{N} = e\sum^{t}_{i= 1} \theta_{i} $, where  and $ t = |G : I_{G}(\theta) |$. Also $ \theta(1) | \chi(1) $ and 	$ \chi(1)/\theta(1) ||G/N| $.
\end{lemma}

\begin{lemma}(see \cite[Theorem E]{4 D2}) \label{distinct}
	Let $ N \trianglelefteq G $, where $ N $ is a $ p $-group  and $ G/N $ is  solvable. Let
	$ \theta \in  \Irr(N) $ be  invariant in $ G $, and assume that   the members of $ \Irr(G |\theta ) $ have distinct degrees. Then, $ G $ is a $ p $-group and 	
	$ \Irr(G |\theta) = \{\chi\} $.  In particular,  $\chi(1)= \sqrt{|G|/|N|}$. %for some prime $p$.
\end{lemma}
\begin{lemma}\label{abeliana}
	Let $G$ be a finite $T'_k$-group, where $k\geq 1$ is an integer. Then $G$ is abelian if and only if  $G$ is an elementary abelian $p$-group, for some prime $p$; or $G\cong C_4$.
\end{lemma}
\begin{proof}
	Since the 'if' part of the statement is clear, we only prove the 'only if' part. Note that for every $\lambda\in\Irr(G)$, $\cod(\lambda)=o(\lambda)$, where $o(\lambda)$ is the order of $\lambda$ in $\Irr(G)\cong G$. So, $G$ does not have a homomorphic image isomorphic to $C_p\times C_q$, for distinct primes $p$ and $q$. Thus, $G$ is a $p$-group, for some prime $p$. First assume that $p\not =2$. If $G$ has a factor group isomorphic to $C_{p^2}$, then $m'_G(p)$ and $m'_G(p^2)$ are greater than one, a contradiction. So $G$ is an elementary abelian $p$-group. Now assume that $p=2$. If $G$ contains an element of order greater than 2, then by easy calculation we can see that $G\cong C_4$. So in this case either $G$ is an elementary abelian 2-group or $G\cong C_4$, and the result follows.
\end{proof}
\begin{lemma}\label{small}
	Let $G$ be a finite group and $K$ be  a nontrivial   normal subgroup of $G$. In addition, let  $\chi\in \Irr(G|K)$,  $\delta\in \Irr(K)$ and $\nu\in \Irr(I_G(\delta)|\delta)$.
	\begin{itemize}
		\item[(a)] There exist $\lambda\in \Irr(K)$, $\theta\in\Irr(I_G(\lambda)|\lambda)$,  such that $\cod(\chi)=|I_G(\lambda)|/(|\ker(\chi)|\theta(1))$ and also $|I_G(\delta)|/(|\ker(\nu^G)|\nu(1))\in \cod(G)$.
		\item [(b)] If $\ker(\chi)\cap K=1$, then $|K|$ divides $\cod(\chi)$.
		\item[(c)] If	 $\chi$ is  faithful, then  $\cod(\chi)\not =p$,  where $p$ is a  prime divisor of $|G|/|K|$.
	\end{itemize}
\end{lemma}
\begin{proof} (a) It is an immediate consequence of Clifford's theorem, \cite[Theorem 6.11]{isaac}. \\
	(b) First, note that we can consider $\chi$ as an irreducible character of $\overline{G}=G/\ker (\chi)$, and adopt the bar convention for this quotient group. Since $K\cap \ker(\chi)=1$, we have  $\overline{K}=K\ker(\chi)/\ker(\chi)\cong K$. Now, using part (a), we have  $$\cod(\chi)=\frac{|I_{\overline{G}}(\lambda)|}{\theta(1)}=|\overline K|\frac{|I_{\overline{G}}(\lambda)|}{|\overline K|\theta(1)}=| K|\frac{|I_{\overline{G}}(\lambda)|}{|K|\theta(1)},$$  for some $\lambda\in \Irr(\overline{K})$ and  $\theta\in \Irr(I_{\overline{G}}(\lambda)|\lambda)$.  By,  \cite[Theorem 6.11]{isaac}, $\theta(1)$ is a divisor of $|I_{\overline{G}}(\lambda)|/|K|$, so $|K|$ divides $\cod(\chi)$.\\
	(c) Now, assume $\ker(\chi)=1$, and $\cod(\chi)=p$, for some prime $p$ dividing $|G/K|$. Hence, using (a) and (b), we have  $|K|=p$ and  there exist $\lambda\in \Irr(K)$, and  $\theta\in\Irr(I_G(\lambda)|\lambda)$,  such that   $|I_G(\lambda)/K|=\theta(1)$. It follows that  $K\leq {\bf Z}(P)$, for every $P\in \Syl_p(G)$. On the other hand, the argument in  \cite[page 84]{isaac} implies that  $|I_G(\lambda)/K|=\theta(1)=1$. Then,   $I_G(\lambda)=K$, which is not possible, as $P\leq I_G(\lambda)$.
\end{proof}
%\begin{remark}
%	If $K$ is the only minimal normal subgroup of $G$, then $|I_{G}(\lambda)|/\theta(1)\in \cod(G)$, for all $\lambda \in \Irr(K)$ and $\theta\in \Irr(I_G(\lambda)|\lambda)$.
%\end{remark}
%\begin{proof}
%	In this case $\ker(\chi)=1$ for all $\chi\in \Irr(G|K)$.  Then by Lemma \ref{small} part(a), we get the result.
%	\end{proof}
\begin{lemma}\label{small2}
	Let $G$ be a finite solvable  $T_k'$-group and $K$ be an abelian  normal subgroup of $G$ of order $r^{a}$, for some prime $r$ and  positive integers $a$ and $k$. Assume, also  $m'_{G}(n)>1$, for some integer $n$.
	\begin{itemize}
		\item[(a)]  Let $\lambda\in\Irr(K)$. Assume  $\cod(\chi)\not =n$ and $\chi$ is  faithful,   for all   $\chi\in \Irr(G|\lambda)$.  Then   there exists integer $b$, such that   $I_G(\lambda)/K$ is a group of order $r^{2b}$.
		\item[(b)] Assume  $\cod(\chi)\not =n$ and $\chi$ is  faithful  for all   $\chi \in \Irr(G|K)$. Then,   Sylow $r$-subgroups of $G/K$ have order $r^{2c}$ for some integer $c$, and Hall $r'$-subgroups of $G/K$ act on $K$, Frobeniusly.
		
	\end{itemize}
\end{lemma}
\begin{proof}
	(a) As $G$ is a $T_k'$-group, and $\cod(\chi)\not =n$, for all $\chi\in \Irr(G|\lambda)$,  using Lemma \ref{small}(a), it follows that  the members of $\Irr(I_G(\lambda)|\lambda)$ must have distinct degrees. Now using Lemma \ref{distinct}, we have  $|I_G(\lambda)|/|K|=r^{2b}$, for some integer $b$.\\
	(b) As $K$ is a  normal subgroup of $G$, then for each $R\in \Syl_r(G)$, we have $ K\cap {\bf Z}(R)\not =1$. Therefore, as $K\cong \Irr(K)$, there exists $1\not =\lambda\in \Irr(K)$ such that   $I_G(\lambda)/K$, contains $R/K$. Now, using part (a), $|R/K|$ must be an even power of $r$. Also $I_G(\lambda)/K$ does not contain any $r'$-element for each nontrivial $\lambda\in \Irr(K)$,  so it follows that Hall $r'$-subgroups of  $G/K$ act on $K$, Frobeniusly.
\end{proof}

\section{Non-solvable $T'_k$-groups}
In this section, we aim to characterize non-solvable $T'_k$-groups, where $k\geq 1$ is an integer, via the following lemmas.\\

\begin{lemma}\label{simple}
	Let $S$ be a finite non-abelian simple  $T'_k$-group, for some integer $k\geq 1$. Then $S\cong {\rm PSL}_2(5)$ or ${\rm PSL}_2(7)$.
\end{lemma}

\begin{proof}
	Since $S$ is simple, for every nontrivial $\chi\in \Irr(S)$, $\ker(\chi)=1$ and so $\cod(\chi)=|G|/\chi(1)$. So $S$ is a $T_k$-group and  using \cite[Theorem 3.2]{hung} the result follows.
\end{proof}

%\begin{lemma}\label{simple}
%	Let $G$ be a finite abelian $T'_k$-group, for some integer $k\geq 1$. Then either $G\cong C_4$ or $G$ is an %elementary abelian $p$-group for some prime $p$.
%\end{lemma}
%	Note that for every $\chi\in \Irr(G)$, we have $\cod(\chi)=o(\chi)$, where $o(\chi)$ is the order of $\chi$ in %$\Irr(G)\cong G$. So it easy to get the result.
%\end{proof}

\begin{lemma}\label{solvability1}
	Let $G$ be a finite group and $M$ be a non-abelian minimal normal subgroup of $G$ such that $G/M\cong \PSL_2(5)$ or $\PSL_2(7)$. Then $G$ is not a $T'_k$-group, for some integer $k\geq 1$.
\end{lemma}
\begin{proof}
	We argue by the contradiction and suppose that $G$ is a $T'_k$-group, for some integer $k\geq 1$. By Lemmas \ref{ somelemmas} and \ref{d2} we get that, $M$ contains an irreducible character $\lambda$ which is extendible to $G$. Since $G/M$ contains two distinct irreducible characters whose degrees are $3$, then according to Gallagher's theorem (\cite[Corollary 6.17]{isaac}), there are $\chi_1, \chi_2\in \Irr(G|\lambda)$ whose   degrees are  $3\lambda(1)$.   We show that  $\ker(\chi_i)=1$, for each $i\in \{1,2\}$. If $\ker(\chi_i)\not= 1$ for some $i\in \{1,2\}$, then $G=\ker(\chi_i)\times M$, implying that $M$ is a non-abelian simple group. Since $G/\ker(\chi_i)$ is a $T'_l$-group, for some integer $l$, then by Lemma \ref{simple}, it should be isomorphic to $\PSL_2(5)$ or $\PSL_2(7)$. Hence  $G\cong \PSL_2(q)^2$  for some $q\in \{5,7\}$ or $G\cong \PSL_2(7)\times \PSL_2(5)$.  In all cases $G$ has $4$ distinct faithful characters of degree $9$, and $2$ distinct irreducible characters with  degree $3$ and kernel $M$, implying that   $m'_G(|G|/9)= 4$ and $m'_G(|G|/(3|M|))= 2$, which is a contradiction.   So $\ker(\chi_1)=\ker(\chi_2)=1$ and $\cod(\chi_1)=\cod(\chi_2)=|G|/(3\lambda(1))$. On the other hand, $G/M$ has two distinct irreducible characters with the same codegree  $|G|/(3|M|)$ which is  different from  $|G|/(3\lambda(1))$, causing a contradiction.
\end{proof}

\begin{lemma}\label{key2}
	Let $G$ be a finite group and $M$ be a  minimal normal subgroup of $G$ such that $G/M\cong \PSL_2(5)$ or $\PSL_2(7)$. Then $G$ is not a $T'_k$-group, for some integer $k\geq 1$.
\end{lemma}
\begin{proof}
	Note that by  Lemma \ref{solvability1}, $M$ is abelian. 	Let $G$ be a counterexample. First we prove that $G$ is perfect.
	If $G'\cap M=1$, then $G=G'\times M$. Let $\lambda\in \Irr(M)$ be a nontrivial irreducible character and $\chi_i\in \Irr(G')$,  for $i\in \{1,2\}$, be characters with degree $3$. Then, $\ker(\chi_i\times \lambda)=1$, for $i=1,2$, and so we get  a contradiction by our assumption as $m'_G(|G|/(3|M|))>1$ and $m'_G(|G|/3)>1$. Therefore, $M\leq G'$ and $G$ is perfect. Let $1_M\not=\lambda\in \Irr(M)$ and $\chi\in\Irr(G|\lambda)$. If $\ker(\chi)\not=1$, then $G=\ker(\chi)\times M$, which is a contradiction, as $G$ is perfect. Hence all characters in $\Irr(G|M)$ are faithful.  Then, by Lemma \ref{small}(b),  $|M|$ is a divisor of $\cod(\chi)$.

	We remark that $C_G(M)=M$, since otherwise $M\leq \mathbb{M}(G/M)$, where $\mathbb{M}(G/M)$ is the Schur multiplier of $G/M$, and $G\cong \SL_2(q)$, for $q\in \{5,7\}$. It is a contradiction, as $G$ is a $T_k'$-group.

	Note that $|G|/(3|M|)\in\cod(G/M)$ has nontrivial multiplicity. We claim that $\cod(\chi)\not=|G|/(3|M|)$, for every $\chi\in\Irr(G|\lambda)$.  Suppose $\cod(\chi)=|G|/(3|M|)$, for some $\chi\in\Irr(G|\lambda)$ and $\lambda\in \Irr(M)$. Since $M$ is an elementary abelian group, it follows that $|M|=2, 4, 5, 7$ or $8$.  As $C_G(M)=M$,  by the normalizer-centralizer theorem, the only possibility is   $G/M\cong \PSL_2(7)$ and $|M|=8$. Therefore, by Lemma \ref{small}(a),   $\cod(\chi)=|G|/(3|M|)=8\times 7= |M|\times |I_G(\lambda)|/(|M|\theta(1))$, for some $\theta\in \Irr(I_G(\lambda)|\lambda)$. Note that, $A=I_G(\lambda)/M$ is a subgroup of $\PSL_2(7)$. Hence, looking at the subgroups of $\PSL_2(7)$  either $|A|=7$  or $|A|=21$, implying that $8=|M|\leq  |G:I_G(\lambda)|\in \{8,8\times 3\}$, which is impossible, by Remark \ref{orbit}. Therefore, all characters in $\Irr(G|M)$ have distinct codegrees and so different degrees. As, $C_G(M)=M$, we deduce that $I_G(\lambda)\not =G$, for every $1\not =\lambda\in \Irr(M)$. So $I_G(\lambda)/M$ is  a proper subgroup of $\PSL_2(5) $ or $\PSL_2(7)$ and so it is solvable, for every $1\not = \lambda\in \Irr(M)$. Now using Lemma \ref{small2}(a), we get that $I_G(\lambda)/M=1$ or $C_2^2$, for all $1\not=\lambda \in \Irr(M)$. Therefore, applying Lemma \ref{distinct}, $\cd(G|M)\subseteq \{30, 60\}$ when $G/M\cong \PSL_2(5)$, and  $\cd(G|M)\subseteq \{84,168\}$ when $G/M\cong \PSL_2(7)$. Hence $\cd(G/M)\cap \cd(G|M)=\emptyset$.
	Thus,  $G$ is a $T_2$-group, and using \cite{hung} we get  a contradiction.
\end{proof}

{\bf Proof of the Main Theorem case (a):}
Let $G\cong \PSL_2(5)$ or $\PSL_2(7)$. Then for every nontrivial $\chi\in \Irr(G)$, $\ker(\chi)=1$ and so $\cod(\chi)=|G|/\chi(1)$.  By \cite[Theorem 3.2]{hung},  $G$ is a $T_2$-group and so $G$ is a $T_2'$-group. 
Now,  we concentrate on the 'only if' part.  Let $G$ be a counterexample with minimal order. Let $M$ be the last term of the derived series of $G$, and $N$ be a normal subgroup of $G$ contained in $M$ such that $M/N$ is a chief factor of $G$.

First, we claim that $G$ is perfect. Suppose false. Hence $M$ is a proper subgroup of $G$. Note that, $G/N$ is a finite non-solvable $T'_l$-group, for some integer $l\geq 1$. If $N\not=1$, then by the minimality of $|G|$, $G/N\cong \PSL_2(5)$ or $\PSL_2(7)$, and consequently  $G=M$, a contradiction. Hence, $N=1$.  Let $C=C_G(M)$. If $C\not= 1$, then by the minimality of $|G|$ we deduce that $M\cong G/C\cong \PSL_2(5)$ or $\PSL_2(7)$ and so $G\cong M\times C$. Since $C$ is solvable, it has a nontrivial linear character, say $\zeta$. Let also $\chi_1$ and $\chi_2$ be two distinct irreducible characters of $M$ of degree $3$. Then $\chi_1\times \zeta$ and $\chi_2\times \zeta$ are two irreducible characters with the same degree $3$ whose kernels are isomorphic to $\ker(\zeta)$. Therefore, they have the same codegree $|G|/(3|\ker(\zeta)|)$ and it is different from $|G|/(3|C|)=\cod(\chi_i\times 1)$ for $i\in \{1,2\}$, a contradiction. Hence $C=1$ and $G$ is a subgroup of $\Aut(M)$.
%By Lemmas \ref{ somelemmas} and \ref{d2} we get that, if $M \not \cong \PSL_2(3^f)^{\alpha}$ for some integers  $\alpha$ and $f\geq 2$, then  $M$ has two nontrivial characters $\theta_1$ and $\theta_2$ with distinct degrees and both are extendible to $G$. Now by Gallagher's theorem (\cite[Corollary 6.17]{isaac}),  contains distinct irreducible characters $\lambda_1$, $\lambda_2$, $\psi_1$ and $\psi_2$ such that  $\psi_1(1)=\psi_2(1)=\theta_1(1)$ and $\lambda_1(1)=\lambda_2(1)=\theta_2(1)$. Since $M$ is the only minimal normal  subgroup of $G$,  $\ker(\chi)=1$ for all  the characters $\chi\in \{\lambda_1, \lambda_2, \psi_1,\psi_2\}$. Thus,  $\cod(\theta_i)$ does not have a trivial multiplicity for $i=1,2$, and  we get a contradiction. So we may assume $M\cong \PSL_2(3^f)^{\alpha}$, for some integers $\alpha$ and $f\geq 2$.

Using Lemmas \ref{ somelemmas} and \ref{d2}, there exists $\lambda\in \Irr(M)$ which is extendible to $G$. We claim that  $G/M$ is a $D_0$-group or an abelian group.  Otherwise, $G/M$ possesses at least  two distinct nonlinear irreducible characters with the same degree $d>1$, and also since $G/M$ is solvable, it has  at least two distinct linear characters. So by Gallagher's theorem, there exists  four  distinct characters  $\lambda_1$, $\lambda_2$, $\psi_1$ and $\psi_2$ in $ \Irr(G|\lambda)$, such that $\lambda_1(1)=\lambda_2(1)=\lambda(1)$ and   $\psi_1(1)=\psi_2(1)=\lambda(1)d$. Note that,  as $M$ is the only minimal normal subgroup of $G$, $\ker(\chi)=1$, for all $\chi\in \{\lambda_1, \lambda_2, \psi_1, \psi_2\}$. Hence, $m'_G(\cod(\lambda_1))>1$  and $m'_G(\cod(\psi_1))>1$, implying that   $G$ is not a $T'_k$-group, which is a contradiction. Therefore, $G/M$ is a $D_0$-group or an abelian group, as desired. Now we claim that $|G/G'|=2 $. Otherwise, there exist at least two distinct nontrivial irreducible characters in $\Irr(G/G')$ with the same codegree. On the other hand, by Gallagher's theorem,  there exist at least two distinct  extensions of $\lambda$ to $G$ with the same codegree, which causes a contradiction. So $G/G'$ is a cyclic group of order $2$. Therefore, looking at the classification of the $D_0$-groups (see \cite{D_0}), we get that the only possibilities  for $G/M$ is being isomorphic to either $S_3$ or $C_2$. Recall that, by Gallagher's theorem,   the multiplicity of $\cod(\chi)$, for some $\chi\in\Irr(G|\lambda)$ is greater than one, where $\chi$ is an extension of $\lambda$. Noting  that $\ker(\psi)=1$, for all $\psi\in \Irr(G|M)$, we deduce   the multiplicity of every character degree in $\cd(G|M)\setminus \{\lambda(1)\}$ must be trivial, as $G$ is a $T_k'$-group, for some integer $k\geq 1$. As $\cd(G/M)\subseteq \{2\}$, we deduce that  $\cd(G|M)\cap \cd(G/M)=\emptyset$. Then,  we conclude that $G$ should be a $T_k$-group, for some integer $k\geq 1$. Now, looking at the structure of non-solvable  $T_k$-groups (see \cite{hung}), we get a contradiction.$\Box$

Hence $G=M$ is perfect. Let $K$ be a maximal normal subgroup of $G$. Then $G/K$ is a simple $T'_r$-group, for some integer $r\geq 1$. Lemma \ref{simple} implies that $G/K \in \{\PSL_2(5),\PSL_2(7)\}$.  If $K=1$, then we are done, so we may assume $K$ is not trivial. Let $L$ be a normal subgroup of $G$ contained in $K$ such that $K/L$ is a chief factor of $G$. Then we work on $\overline{G}=G/L$ and we use the bar convention for this quotient group. As, $\overline{G}/\overline{K}\in \{\PSL_2(5),\PSL_2(7)\}$,  by Lemma \ref{key2}, we get  a contradiction.

\section{Solvabel $T'_k$-group}
In this section we aim to study solvable $T'_k$-groups for some integer $k\geq 1$,  and prove Case (b) of the main theorem,  through Lemmas \ref{nilpotent}-\ref{suz}.

\begin{lemma}\label{nilpotent}	Let $G$ be a  finite non-abelian $p$-group, for some prime $p$. Then  $G$ is a  $T'_k$-group for some  integer $k\geq 1$, if and only if $p=2$ and
	$G$ is an extra-special group.
\end{lemma}
\begin{proof}
	First, we prove that $G$ is a rational group and $p=2$.	
	Let $\chi \in \Irr(G| G')$ and $\epsilon$ be a $| G|$th root of unity and $\delta\in  $Gal$(\mathbb{Q}(\epsilon)/\mathbb{Q})$. Then $\chi^{\delta}\in \Irr(G)$ and $\chi^{\delta}(1)=\chi(1)$.
	As $|\ker(\chi^{\delta})| =| \ker(\chi)|$,   we  have $\cod(\chi^{\delta}) =\cod(\chi)$.  We show that  $|G|/|\ker(\chi)G'|$ is not trivial.  Suppose false,  and take $ K$ to be a normal subgroup of $G$, contained in $ G'$, such that $G'/K$ is a cheif factor of $G$ and $\ker(\chi)\cap G'\leq K$.  Hence, 	
	$$\frac{G}{K}\cong \frac{G'}{K}\times \frac{\ker(\chi)K}{K}$$
	is an abelian group, implying that $G'\leq K$, a contradiction.    Therefore by Lemma \ref{small}(c), $\cod(\chi)\not =p$.   So, as $ G$ is a $T'_k$-group, for some integer $k$,  and $m'_G(p)>1$,  we get that $\chi^{\delta}=\chi$. Hence $\chi(x)$ is rational for all $x\in G$. So, all nonlinear irreducible characters of $G$ are rational-valued,
	implying that  $|G|$ is even  by \cite[(3.16), p. 46]{isaac}, and it follows that $p = 2$. Note that by Lemma \ref{abeliana}, either  $G/G'\cong C_4$ or $G/G'$ is an elementary abelian  $2$-group. Since $G$ is a $2$-group, the first case is not possible and so linear  characters of $G$ are also rational-valued.  Hence  $ G$ is a rational group, as wanted.

	Assume $K$ is a normal subgroup of $G$, contained in $G'$, such that $G'/K$ is a chief factor of $G$. Write $\overline G=G/K$ and adopt the bar convention for this quotient group.
	Remark that $\overline G'$ is a  central subgroup of $\overline G$ of order $2$.  Then,  ${\bf Z}(\overline G)\cong \overline N\times \overline G'$, where $\overline N$ is an elementary abelian group, by \cite[Theorem 3.6]{dim}.
	
	We prove  that ${\bf Z}(\overline G/\overline N)=\overline G'\overline N/\overline N$. Suppose false, then there exists an element  $x$ out of $\overline G'\overline N$, such that for all elements $y\in \overline G$, $[x,y]\in \overline N$. On the other hand $[x,y]\in \overline G'$, implying that $[x,y]=1$ and so $x\in {\bf Z}(\overline G)=\overline G'\overline N$, a contradiction, as wanted.	
	Therefore, $\overline G/\overline N$ is an extra-special group. Now, we aim to prove that $\overline N=1$.  Suppose false. Let $\overline N_1, \overline N_2$ be two different subgroups of ${\bf Z}(\overline G)$ of order $|\overline N|$, not containing  $\overline G'$. Then, replacing $\overline N$ with each one of $\overline N_i$'s, for $i=1,2$, similar to the above discussion, we have $\overline G/\overline N_i$, for $i=1,2$, are extra-special groups with the same order. Therefore, $\overline G$ has two distinct nonlinear  irreducible characters with the same codegree different from $2$, a contradiction. Thus, $\overline N=1$.  Hence $\overline G$ is an extra-special group.	
	If $K$ is trivial, we are done. Hence, we may assume $K$ is a minimal normal subgroup of $G$. Then $K\leq {\bf Z}(G)$.
	
	Note that,  $K={\bf Z}(G)$, as otherwise ${\bf Z}(G)=K\times T$, where $TK/K= {\bf Z}(\overline G)$ (as ${\bf Z}(G)/K\leq {\bf Z}(G/K) $ and $|{\bf Z}(G/K)|=|G'/K|=2$ ), hence replacing $K$ by $T$ in the above discussion, we have  $G/T$ and $G/K$ are  isomorphic and so $G$ has two distinct nonlinear irreducible characters with the same codegree, a contradiction. Then, ${\bf Z}(G)=K$ and  $\chi$ is faithful for all $\chi\in \Irr(G|K)$,  yielding that $|G/K|$ is an square,  by Lemmas \ref{small}(c) and \ref{small2}(b), a contradiction. So $K=1$, as wanted.
	% Then   $G$ is a $D_0$-group and so $G$ is an extra-special group which is impossible and so $K$ is a $2$-group.  Therefore $K< {\bf Z}(G)$. If $\ker(\chi)\not =1$, for some $\chi\in \Irr(K)$, then ${\bf Z}(G)=G'\cong C_2^2$. Then $m'_G(\sqrt(|G|)/4)>1$, a contradiction, so $K=1$.   Hence we may assume $G/K\cong T\times C_2$, where $T$ is an extraspecial group. With the same manner we prove that $|K|$ is a $2$-group and $\ker(\chi)=1$ for all characters in $\Irr(G|K)$. Therefore ${\bf Z}(G)=K$ and $G$ is a $D_1$-group and the only $D_1$-group is the one satisfing in the part(7) of the main Theorem. Now let $K/T$ is the cheif factor of $G$. Now we are going to prove that $T$ is trivial....
\end{proof}

\begin{lemma}\label{key1}
	Let $G$ be a finite solvable  $T'_k$-group, for some integer $k\geq 1$,  and $G/G'\cong C_p^{\alpha}$,  for some prime $p$ and integer $\alpha$. Assume also $G'$ is a minimal normal subgroup of $G$ whose order is co-prime to $|G/G'|$. Then one of the following occurs:
	\begin{itemize}
		\item[(1)]  $G\cong C_2\times S_3$;
		\item[(2)] $G\cong D_{2q}$ for some prime $q\geq 3$; or
		\item [(3)] $G$ is a Frobenius group  $C_2^{\beta}\rtimes C_{2^{\beta}-1}$, where $p=2^{\beta}-1$. In particular $\beta$ is odd.
		
	\end{itemize}
\end{lemma}
\begin{proof}
	Let $|G'|=q^{\beta}$, for some prime $q$ and   integer $\beta$ and  $C_G(G')=G'\times C$, where $C$ is an elementary abelian $p$-group. Then $C\leq {\bf Z}(G)$ and $G=C\times T$, where $T$ is a subgroup of $G$ containing $G'$. Remark that $T$ is not isomorphic to $G'$. Note that, $\chi$ is faithful, for all  $\chi\in \Irr(T|T')$, as $\ker(\chi)\leq C_T(T')=T'=G'$ and $\ker(\chi)\cap T'=1$.
	
	Firstly, assume that $|T|_p>2$, then $m'_T(p)>1$. Hence, by Lemma \ref{small}(c),  we get that  all characters in $\Irr(T|T')$ have distinct degrees and so $T$ is a $D_0$-group. Looking at the structure of $D_0$-groups we get that $T$ is a Frobenius group, with cyclic Frobenius complement of order $p$ and  
	%for all $\lambda \in \Irr(T')$, $\lambda^{T}\in \Irr(T)$ and so all of the nontrivial irreducible  characters in $T'$ are conjugate, as all characters in $\Irr(T|T')$ have distinct degrees, implying that 
	$p=q^{\beta}-1$.  This is not possible, unless $q=2$ and $\beta$ is odd. Hence $p\not= 2$. Now we show  that in this case $C$ is trivial. Assume, $C=C_p^{\alpha_1}$ and $\alpha_1\geq 1$. Since $p\not =2$, we deduce that   $|\Irr(C)\setminus\{1_C\}|> 1$. Let    $\lambda_1, \lambda_2\in \Irr(C)$ be  two distinct nontrivial characters and $\chi\in \Irr(T)$ be  the only nonlinear irreducible character  of $T$. Then $|\ker(\lambda_1\times \chi)|=|\ker(\lambda_2\times \chi)|=|C|/p$ and $\lambda_1\times \chi(1)=\lambda_2\times \chi(1)=\chi(1)$,  so $\cod(\lambda_1\times \chi)=\cod(\lambda_2\times \chi)$. Note that $\cod(\lambda_1\times \chi)$  is not equal to $p$, by Lemma \ref{small}(c),  a contradiction, as  both multiplicities of $\cod(\lambda_1\times \chi)$ and $p$ are not
	trivial. Hence $C=1$. So $G$   is the described group in case (3) of the Lemma.
	
	Now we  assume $|T|_p=p=2$. Recall that, $C_{T}(T')=T'$ and $T'$ is a  minimal normal  subgroup of $T\cong G/C$,  where $T$ is a Frobenius group whose Frobenius complements are  cyclic groups of order $2$. Therefore, $T\cong G/C\cong D_{2q}$. First assume $q> 3$.   We  prove that $C$ is trivial. Suppose false. In this case $\Irr(G/C)$ contains two   distinct nonlinear irreducible characters $\chi_1$ and $\chi_2$. Let $\lambda\in \Irr(C)$ be a nontrivial irreducible character. Then, $|\ker(\chi_1\times \lambda)|=|\ker(\chi_2\times \lambda)|=|\ker(\lambda)|=|C|/2$ and so  $\cod(\chi_1\times \lambda)=\cod(\chi_2\times \lambda)=2\cod(\chi_1)$, hence $m'_G(2\cod(\chi_1))>1$. Noting that $m'_{T}(\cod(\chi_1))=(q-1)/2>1$, we get a contradiction. Therefore $C=1$. So $G\cong D_{2q}$, a described group in case (2) of the lemma.
	
	The last possibility is that $q=3$ and $G\cong C_2^{\alpha-1}\times D_6$. Hence, in this case $m'_G(6)=2^{\alpha-1}-1$ and $m'_G(2)=2^{\alpha}-1$. Therefore, either $\alpha=1$ or $\alpha=2$, which means that either  $G\cong C_2\times S_3$ or $G\cong D_6$, which are the groups in case (1) and (2) of the lemma.
\end{proof}
% By minimality of $M$, either $M$ is a subgroup of Schur multiplier of $\rm{PSL}_2(5)$ or  for all nontrivial $\lambda\in Irr(M)$, $I_G(\lambda)<G$. The former case yields that $G\cong SL_2(5)$ and we get a contradiction. So the latter case occurs.   Let $|M|=p^a$ for some prime $p$. We claim that $p=2$. On the contrary,  assume $p\not =2$. 	

%Looking at the subgroups of $A_5$, by  \cite[  Theorem 6.26 and Corollary 11.22]{issac},  we deduce that  $ \lambda $  is extendible to  $ \lambda_{0} \in \Irr(I_{G}(\lambda)) $  for every   $1_{N} \neq\lambda \in \Irr(N) $ and so by the solvability of the proper subgroups of $PSL_2(5)$ we get a contradiction, except fob the case $I_G(\lambda)=M$. Therefore $G$ is a Frobenius group, a contradiction. So $p=2$, as desired. With the same discussion as above we get that either $I_G(\lambda)=M$ or $I_G(\lambda)/M=C_2\times C_2$, for all nontrivial $\lambda \in Irr(M)$.
%\end{proof}

\begin{lemma}\label{key3}
	Let  $G$ be a finite solvable $T'_k$-group for some integer $k\geq 1$ and $K< G'<G$ be  a normal series of $G$. Assume also $G/K$ is  isomorphic to  the  Frobenius group $C_{2^{\beta}}\rtimes C_p$, for some integer $\beta$ and a prime number $p=2^{\beta}-1$. Then, $K=1$.
\end{lemma}
\begin{proof}
	Let $G$ be a counterexample with minimal order.  Therefore, we may assume  $K$ is a minimal  normal  subgroup of $G$. 	
	First of all,  we show  that   $\ker(\chi)=1$, for all $\chi\in \Irr(G|K)$. Suppose false. Then,  there exists $\chi\in \Irr(G|K)$ such that $\ker(\chi)K/K$ is a nontrivial normal subgroup of $G/K$, implying that $G'\cong K \times T$, where $T\cong \ker(\chi)$ (Note that $\ker(\chi)\not \cong G/K$, as   $|G/G'|$ is a prime). Now, we work on group $\overline{G}=G/T$, and we adopt bar convention for this quotient group. Note that $\overline{K}= \overline{G}'$. As $|G/G'|$ is prime, we deduce that  $(|\overline{G'}|,|\overline{G}/\overline{G}'|)=1$. So the hypothesis of the Lemma \ref{key1}, is provided for $\overline{G}$. Applying Lemma \ref{key1}, $|K|=|\overline{G}'|=2^{\beta}$. Hence both groups $G/K$ and $G/T$ have  faithful irreducible characters with the same codegree $2^{\beta}$ and so $m'_{G}(2^{\beta})>1$, which is a contradiction, as $m'_G(p)>1$. So $T=1$, as wanted.

	% remark that $\cod(\chi)=|I_G(\lambda)|/\theta(1)=|K||I_G(\lambda)|/|K|\theta(1)$, where $\chi\in \Irr(G|\lambda)$   and $\theta \in \Irr(I_G(\lambda)|\lambda)$, for some $1\not = \lambda \in \Irr(K)$. Therefore, as $\theta(1)$ divides $|I_G(\lambda)|/|K|$, we get that $|K|$ divides $\cod(\chi)$.
	
  Hence, by Lemmas \ref{small}(c) and \ref{small2}(b), we get that   $G/K$ acts on $K$ Frobeniusly, which is a contradiction, as $G/K$ is a Frobenius group itself and it is well-known that a complement of a Frobenius group is not Frobenius.
\end{proof}

\begin{lemma}\label{c2s3} Let $G$ be a finite solvable $T'_k$-group for some integer $k\geq 1$,  and $K<G'<G$ be a normal series of $G$, such that $G/K\cong C_2\times S_3$. Then $K=1$.
\end{lemma}
\begin{proof}
	Let $G$ be a counterexample with minimal order. Then, we may assume $K$ is a minimal  normal  subgroup of $G$ of order $p^{\alpha}$, for some prime $p$ and integer $\alpha$.  Let $T=\ker(\chi)$ for some $\chi\in \Irr(G|K)$.   Then, $T$ is isomorphic to one of the groups in  $ \{S_3, C_6,  C_2, C_3, 1\}$.
	
	Let $p\not =2$. Assume $T\in\{S_3, C_6, C_3\}$. Then $G'= T_0\times K$, where $T_0=T\cap G'\cong C_3$. Then,   $G/T_0$ satisfies the hypothesis of  Lemma \ref{key1}. Therefore, $G/T_0\cong C_2\times S_3$ and $K\cong C_3$. Note that, $m'_G(2)>1$ and both $G/T_0$ and $G/K$ contain faithful irreducible characters whose codegrees are $3$. So $m'_G(3)>1$, a contradiction.   Therefore,  $T\cong C_2$ or $1$. Set $S=\Pi_{\chi_0\in \Irr( G| K)}\ker(\chi_0)$.  Now, we consider the group $\overline{G}=G/S$ and we adopt the bar convention for this quotient group. Then  $\ker(\chi_1)=1$, for all $\chi_1\in \Irr(\overline{G}|\overline{K})$.  Note that $\overline G/\overline K\cong S_3$ or $C_2\times S_3$.   By Lemmas \ref{small}(c) and \ref{small2}(b),   $\overline{G}/\overline{K}$ acts on $\overline{K}$, Frobeniusly, a contradiction, as $\overline G/\overline K\cong S_3$ or $C_2\times S_3$.
	%Therefore $p\in \{2,3\}$.
	%Let $p=3$. Obviously,  $K\subseteq {\bf Z}(G')$ and hence $G'$ is an abelian $3$-group.  Let $\chi \in \Irr(G|K)$. Then $\ker(\chi)\in \{1, C_3, C_2, S_3, C_6\}$. If $\ker(\chi)\cong C_3, S_3$, then $\overline {G}=G/\ker(\chi)$ satisfies Lemma \ref{key1}, and $K\cong C_3$ and $G\cong T \times C_2$, where $T$ is a Frobenius group of order $18$, which is impossible, as $m'_G(3)>1$. Then for all $\chi\in \Irr(G|K)$, $\ker(\chi)\cong C_2$ or 1.  Let   $\ker(\chi)\cong C_2$ for some $\chi \in \Irr(G|K)$ and considering $\overline{G}=G/\ker(\chi)$, we see that $\ker(\overline{\chi})=1$, for all $\overline{\chi}\in \Irr(\overline{G}|\overline{K})$. Then $\Irr(I_{\overline{G}}(\lambda))$ must have distinct degrees, for all $\lambda \in \Irr(\overline{K})$, implying that $I_{\overline{G}}(\lambda)=\overline{K}$, for all $\lambda\in \Irr(\overline{K})$. which implies a contradiction. So $\ker(\chi)=1$ for all $\chi\in \Irr(G|K)$ and we get a contradiction  arguing in the  same manner.
	
	Finally, let $p=2$.  If $C=C_G(K)$ contains a Sylow $3$-subgroup of $G$, then $K\leq  {\bf Z}(G)$ and so $|K|=2$. Hence, $G$ has a normal abelian subgroup $L$ of order  $12$.  Looking at the classification of groups of order 24, we see  $G\cong D_{24}$ or $G\cong $SmallGroup$(24,8)$. In the first case  there exists $2$ distinct  faithful characters with the same degree $2$, a contradiction, as $m'_G(2)>1$. In the second case $S=$Socle$(G)=C_3\times C_2$,  so  every $\chi\in \Irr(G|\lambda)$ is faithful, and  $I_G(\lambda)/S$ is a group of order $2$, for all $\lambda\in \Irr(S)$ of order $6$. Then by Gallagher's theorem, there is two  distinct faithful characters of degree $2$ above each irreducible character $\lambda\in \Irr(S)$ of order $6$, a contradiction.

	So $C$ does not contain a Sylow $3$-subgroup of $G$ and hence  $T\cong C_2$ or 1.  Let $S=\Pi_{\chi\in \Irr(G|K)}\ker(\chi)$. Note that  $S\cong C_2$ or 1.  Setting  $\overline{G}=G/S$, we see that $\ker(\chi)=1$, for all $\chi\in \Irr(\overline{G}|\overline{K})$. Then by Lemmas \ref{small}(c) and \ref{small2}(a),  $I_{\overline G}(\lambda)/\overline K\cong C_2^2$ or $1$, for all $\lambda\in \Irr(\overline K)$.   Applying Lemma \ref{distinct}, we have $\cd(\overline G|\overline K)\subseteq \{6, 12\}$. Since $\cd(\overline G|\overline K)\cap \cd(\overline G/\overline K)=\emptyset$, then  $\overline G$ must be 
	a $T_1$-group($D_0$-group).
	Looking at the classification of the $D_0$-groups(see \cite{D_0}) we get a contradiction, as $\overline G\cong S_3$ or $C_2\times S_3$.
\end{proof}

\begin{lemma}\label{s41} Let $G$ be a solvable $T'_k$-group for some $k\geq 1$, and $K <G'<G$ be a normal series of $G$ such that $G/K\cong S_3$ and $K$ is a minimal normal  subgroup of $G$ of order $r^{\alpha}$, for some prime $r\not =3$. Then $G\cong S_4$.
\end{lemma}
\begin{proof}
	We show that $C_G(K)=K$. Suppose false, then $G'=T\times K$, where $T\cong C_3$.  As $|G/G'|=2$, we deduce that $r\not =2$. Therefore,   $G/T$ satisfies Lemma \ref{key1}, and so $|K|=r$ is an odd prime other than $3$. By Lemma \ref{key3} we get a contradiction. Hence $C_G(K)=K$. Then $\chi$ is faithful for all $\chi \in \Irr(G|K)$. Remark that, $I_G(\lambda)/K$ does not contain  the Sylow $3$-subgroup of $G/K$,  for any $\lambda\in \Irr(K)$, as otherwise ${\bf Z}(G')\cap K\not =1$, contradicting the fact that $C_G(K)=K$.  Hence, $I_G(\lambda)/K\cong 1 $ or $C_2$,  for every $\lambda\in \Irr(K)$,  implying that $G$-orbit sizes of  nontrivial elements of $K$ are  either $3$ or $6$.    As  Sylow subgroups of $I_G(\lambda)$ are cyclic, $\lambda$ extends to $I_G(\lambda)$,  for all $\lambda \in \Irr(K)$.  If $I_G(\lambda)/K\cong C_2$,  then  $m'_G(|I_G(\lambda)|)>1$ by Lemma \ref{small}(a). Also, if $I_G(\lambda_1)=I_G(\lambda_2)=K$, for two non-conjugate  nontrivial irreducible  characters of $K$, then $m'_G(|K|)>1$.
	
	Let $n_i$ be the number of $G$-orbits of size $i$ of elements of $K\cong\Irr(K)$.  As  $G$ is a $T'_k$-group, by the above discussion, we have  the following possibilities for the  number of $G$-orbit sizes  of  nontrivial elements of $\Irr(K)$: $(n_3, n_6)\in \{(k,0), (k,1), (0,k)\}$, for some integer $k$. The last case yields that $G/K$ acts on $K$ Frobeniusly, which is a contradiction, as $G/K$ is a Frobenius group.
	If the second case occurs, then applying Lemma \ref{size Nq}, we get that $r^{\alpha}-7=n_2(r^{\beta}-1)=3(r^{\beta}-1)$, for some integer $\beta$. Whence, $r^{\beta}(r^{\alpha-\beta}-3)=4$, implying that $|K|=r^{\alpha}=16$. On the other hand, by Lemma \ref{small2}(b) the Sylow $3$-subgroup of $G/K$ acts on $K$, Frobeniusly,  yielding that $G$ has a normal Frobenius subgroup of order $48$, say $L$.  Then $L$ has exactly  5 distinct irreducible  characters of degree $3$. Since $|G/L|=2$, these five characters belong to at least 3 distinct $G$-orbits and the size  of one of these orbits  must be $1$. Hence one of the described  characters extends to $G$ and we have two distinct faithful  irreducible characters with the same codegree $|G|/3$.  If there exist  two $G$-orbits of size $2$ containing one of the described characters, then we have two distinct faithful irreducible characters of codegree $|G|/6$, which is a contradiction. If there exists only  one $G$-orbit with size $2$, then it follows that we have one faithful irreducible character of degree $6$ and $8$ distinct faithful irreducible characters of degree $3$. On the other hand $|G|=96< 6^2+ 8.3^2$, a contradiction. So all of these five  irreducible characters of $L$ belong to $G$-orbits of $G$ of size $1$ and so  all of the faithful irreducible characters of $G$ have the same degree $3$. But we have  $n_6=1$, which means that there is a character $\lambda\in \Irr(K)$, such that $\lambda^{G}\in \Irr(G)$,  a contradiction as $\lambda^G(1)=6$.

	Then the first case occurs, and  applying Lemma \ref{size Nq}, we have $|K|=4$ and $G\cong S_4$, as desired.
\end{proof}

\begin{lemma}\label{s42}Let $G$ be a finite solvable $T'_k$-group, where $k\geq 1$, and $M <G'<G$ be a normal series of $G$ such that $G/M\cong S_4$. Then $M=1$.
\end{lemma}
\begin{proof}
	Let $G$ be a counterexample with minimal order. Then we may assume $M$ is a minimal normal subgroup of $G$ of order $|M|=r^{\alpha}$, for some prime $r$ and integer $\alpha$.
	We  show that $M$ is the only minimal normal subgroup of $G$. Suppose false, then Socle$(G)\not= M$, so Socle$(G)=PM\cong M \times P$, where $PM/M$ is the  Sylow $2$-subgroup of $G'/M$. First, let $r\not =3$.  Looking at $G/P$ and applying  Lemma \ref{s41}, we get that $M \cong C_2^2$. If we take $\lambda\in \Irr(M)$ and $\lambda'\in \Irr(P)$ to be two nontrivial characters, then all ellements in  $\Irr(G|\lambda\times \lambda')$ are faithful. Note that $m'_G(8)>1$. Now, using Lemma \ref{small}(b),  the  codegrees of all characters in $\Irr(G|\lambda\times \lambda')$ are divided by $|PM|=16$. Hence, by Lemma \ref{small2}(a), we get that $I_G(\lambda\times \lambda')=PM$. Note that all such characters must be conjugate, as otherwise we have at least two distinct faithful irreducible characters with the same degree $6$, and so $m'_G(|G|/6)>1$, by Lemma \ref{small}(a), a contradiction.  So, $9=(|P|-1)(|M|-1)=|G:I_G(\lambda\times \lambda')|=6$, a contradiction.  Now, assume $r=3$. Then by \cite{ebrahimi}, $G/P$ is not a $T'_1$-group and so, $m'_G(d)>1$ for some integer $d\not =8$ ($d$ must divide $|G/P|$). Since in  $G/K\cong S_4$, multiplicity of $8$ is not trivial we get a contradiction.  Therefore, our claim is proved and $\chi$ is faithful for all $\chi\in \Irr(G|M)$.

	Now, we show that $\cod(\chi)\not =8$, for all $\chi\in \Irr(G|M)$. Assume $\cod(\chi)=8$, for some $\chi\in \Irr(G|M)$. Then, $8=|M|\times |I_G(\lambda)|/(|M|\theta(1))$, where $1_M\not =\lambda\in \Irr(M)$, $\theta\in \Irr(I_G(\lambda)|\lambda)$ and $\chi\in \Irr(G|\lambda)$, by Lemma \ref{small}(a). So, $M$ is a $2$-group and $P/M\leq I_G(\lambda)/M$, where $P/M$ is the Sylow $2$-subgroup of $G'/M$.
	
	Then, $(|M|, |I_G(\lambda)|/|M|, \theta(1))\in \{(8,1,1), (4,2,1), (4,4,2), (2,8,4) \}$. The first two cases are not possible as $P/M\leq I_G(\lambda)/M$.  The last case is not possible, as by \cite[page 84]{isaac}, $16=\theta(1)^2 \leq |I_G(\lambda)|/|M|=8$. So, the only possibility in this case is $|M|=|I_G(\lambda)/M|=4$ and $\theta(1)=2$. In this case $4=|M|<|G:I_G(\lambda)|=6$, contradicting Remark \ref{orbit}. So our claim is proved and $G$ satisfies hypothesis of Lemma \ref{small2}(b). Therefore $G/M\cong S_4$ acts on $M $ Frobeniusly, which is a contradiction.	
	% Now, let $t=32$, then $G/M\cong$SmallGroup$(96,64)$ and  $(|M|, |I_G(\lambda)/M|, \theta(1))\in \{(32,1,1), (16,2,1),(8,4,1), (4,8,1), (4,16,2),$ $(4,32,4),(2,16,1), (2, 32,2)\}$. First four cases are not possible, as $P/M\subseteq I_G(\lambda)/M$. The fifth and seventh cases do not occur, as $|M|<|G:I_G(\lambda)|=6$.  Last case, is not possible, as $T\subseteq C_G(M)$, for all $T\in \Syl_2(G)$ and, hence  $M\subseteq {\bf Z}(G)$. Therefore, the only possibility is $|M|=4$, $|I_G(\lambda)/M|=32$ and $\theta(1)=4$. Checking by GAP, there is no $T'_k$-group of order $384$.
\end{proof}

\begin{lemma}\label{s3}  Let $G$ be a finite  solvable $T'_k$-group, for some integer $k\geq 1$ and $|G/G'|=2$. Then either $G\cong S_4$;   $G\cong D_{18}$; or $G$ is a Frobenius group  isomorphic to $C_p^{\beta}\rtimes C_2$, for some odd prime $p$ and integer $\beta$.
\end{lemma}

\begin{proof} We aim to prove that either $G'\cong A_4$; $G'\cong C_9$ or   it  is an elementary abelian  $p$-group, for some odd prime $p$ .
	%%%%%%%%%%%%%%%%%%%%%%%%%%%%%%%%%%%%%%%%%%%%%%%%%%%%%%%%%%%
	First, we assume that $G'$ is a $p$-group, for some  odd prime $p$.  We work on $\overline{ G}=G/G'' $ and we use the bar convention for this quotient group.    As $|\overline G/\overline G'|=2$, by Fitting Lemma (see \cite[Theorem 4.34]{isaacs}) we have  $\overline G/\overline G'$ acts on $\overline G'$ Frobeniusly.

	We  prove that $\overline G'$ is an elementary abelian group or a cyclic group of order $9$. Let $\overline  G'\cong \overline T\times \overline L$, where $\overline T$ is a cyclic group of order $p^m$, for some $m>1$. Note that as $\overline G/\overline G'$ acts on $\overline G'$ by inverting every element, then every subgroup of $\overline G'$ is a normal subgroup of $\overline G$. Now, we consider group $\overline G/ \overline L\cong D_{2n}$, where $n=p^m$. Since $\overline  G/\overline  L$ is a $T'_l$-group for some integer $l$, it follows that $n=9$ (notice that  there are $\phi(d)/2$ irreducible characters in $\Irr(D_{2n})$ with the same codegree $d$, for each divisor $d$ of $n$, where $\phi$ is the Euler's totient function).   If $\overline L\not =1$, then $\overline G$ has a homeomorphic image isomorphic to the Frobenius group $(C_3\times C_9)\rtimes C_2$, which is not possible, as in this case $m'_{\overline G}(9)$ and $m'_{\overline G}(3)$ are not trivial. Therefore $\overline L=1$ and $\overline G\cong D_{18}$, as desired. So either $p=3$ and  $\overline G'\cong C_9$ or it is an elementary abelian $p$-group.
	
	Now, we  show that $G''=1$. Let  $G'/G''\cong C_9$ and $G''\not =1$, then  we may assume $G''$ is  a  minimal normal subgroup of $G$  and so $G''\leq {\bf Z}(G')$, a contradiction.  Hence in this case  $G''=1$. 
	So, we assume $G'/G''$ is an elementary abelian group. Contrary to what we want to prove, assume $G''\not =1$ and  take $L$ to be a normal  subgroup of $G$, such that $G''/L$ is a chief factor of $G$. Write $\widetilde G=G/L$ and adopt the tilde convention for $\widetilde G$.  Note that $\widetilde G''\leq {\bf Z}(\widetilde G')$. We claim that $\widetilde G$ is a Frobenius group. Suppose false. Since $\widetilde G/\widetilde G'$ acts on $\widetilde G'/\widetilde G''$, Frobeniusly, we may assume   $\widetilde G/\widetilde G'$ acts on $\widetilde {G}''$, trivially and so $\widetilde G''\leq {\bf Z}(\widetilde G)$, hence $\widetilde G''\cong C_p$.  Let $\widetilde T=\Pi_{\chi\in\Irr(\widetilde G|\widetilde G'')} \ker(\chi)$. Then $\widetilde T$ is a normal subgroup of $\widetilde G$, whose intersection with $\widetilde G''$ is trivial. Note that $\widetilde T\times \widetilde G''$ is a proper subgroup of $\widetilde G'$, so the order of $\widetilde G/\widetilde G''\widetilde T$ is divided by $2p$. Also, $\ker(\chi)=1$, for all $\chi \in \Irr(\widetilde G/\widetilde T|\widetilde G''\widetilde T/\widetilde T)$.
	Note that $\widetilde G/\widetilde G''\widetilde T \not \cong D_{2p}$, as otherwise 	
	$$ \frac{\frac{\widetilde G'}{\widetilde T}}{\frac{\widetilde G''\widetilde T}{\widetilde T}}\cong \frac{\widetilde G'}{\widetilde T\widetilde G''}\cong C_p,$$
	which is not possible.
	Hence  $\widetilde G/\widetilde G''\widetilde T$ is isomorphic to the Frobenius group  $ C_p^{\beta}\rtimes C_2$, for some integer $\beta > 1$,  then $p\in \cod(\widetilde G/\widetilde G''\widetilde T)$ and $m'_G(p)>1$. Therefore, using Lemmas \ref{small}(c) and \ref{small2}(b), $\widetilde G/\widetilde G' $ acts on $\widetilde T \widetilde G''/\widetilde T\cong \widetilde G''$, Frobeniusly, a contradiction. So our claim  is proved and $\widetilde G$ is a Frobenius group.
	
	It is well-known that if the complements  of a Frobenius group are even order, then the Frobenius kernel is abelian. Hence,  $\widetilde G'$ is abelian and $\widetilde G''=1$, which means $G''=1$ as wanted.	
	
	%%%%%%%%%%%%%%%%%%%%%%%%%%%%%%%%%%%%%%%%%%%%%%%%%%%%%%%%%%%%%	
	Finally, assume  $G'$ is not a $p$-group, for some  odd prime $p$.   
	Clearly $ G'/G''$ is not divided by two distinct primes $p <q$, since otherwise $G$ has a homomorphic image isomorphic to $D_{2pq}$, a contradiction, as $m'_{D_{2pq}}(q)>1$ and $m'_{D_{2pq}}(pq)>1$. So, we may assume $\overline G'$ is a $p$-group for some odd prime $p$.
	
		Then define   $\mathcal{S}$ to  be the set of all  pairs $(M,N)$, where $N\leq M$   are normal subgroups of   $G$ contained in $G'$  and $M/N$ is a chief factor of $G$ whose order is not divided by $p$.  Suppose $(M,N)$ is a pair in $ \mathcal{S}$ such that  it has the largest  first component among all members of  $\mathcal{S}$. Remark that  $M\not = G'$, as otherwise, $|G'/G''|$ is divided by two distinct primes.

	Now, we work on $\overline{G}=G/N$ and we use bar convention for this quotient group. Let $\overline{C}=C_{\overline{G}'}(\overline{M})\cong \overline{M}\times \overline{T}$. If $\overline{T}>1$, then $\overline{C}/\overline{T}\cong \overline{M}$ is a $q$-group for some prime $q\not=p$, a contradiction, as $|\overline M|<|\overline C|$.  Since $|G/G'|=2$, we deduce that  $C_ {\overline G}(\overline M)=\overline M$ and hence $\chi$ is faithful, for all $\chi\in \Irr(\overline G|\overline M)$. First, let $|\overline G'/\overline M|>3$. Then,  $\overline G/\overline M$ is isomorphic to the Frobenius group $C_p^{\beta}\rtimes C_2$, for some integer $\beta$, as $|\overline G'/\overline M|$ is a $p$-group (notice that $\beta>1$ when $p=3$). Hence $m'_{G/M}(p)>1$.    
	
	We note that $\cod(\chi)\not =p$, for all $\chi \in \Irr(\overline G|\overline M)$ by Lemma \ref{small}(b). Hence,  $\overline G/\overline M$ acts on $\overline M$, Frobeniusly, by Lemma \ref{small2}(b).  As $\overline G/ \overline M$ is  Frobenius, we get a contradiction.	
	So we assume $G'/M\cong C_3$ and  by Lemmas \ref{s41} and \ref{s42}, $G\cong S_4$ as wanted.
\end{proof}
%##############################}
\begin{lemma}\label{lastone} Let $G$ be a finite  solvable $T'_k$-group for some integer $k\geq 1$, and $K<G'< G$ be  a normal series of $G$, such that $G/K$ is a $p$-group for some prime $p$. Then either $G$ is an extra-special $2$-group, or  $G$ is a Frobenius group isomorphic to $C_3^2\rtimes Q_8$.
	
\end{lemma}
\begin{proof}
	Applying  Lemma \ref{nilpotent} on $G/K$, we have $p=2$. If $G'$ is a $2$-group, we are done, so we may assume  $G'$ is not a $2$-group. Define   $\mathcal{S}$ to  be a set of all  pairs $(M,N)$, where $N\leq M$   are normal subgroups of   $G$ contained in $G'$  and $M/N$ is a chief factor of $G$ whose order is not divided by $2$. Suppose $(M,N)$ is a pair in $ \mathcal{S}$ such that  it has the largest  first component among all members of  $\mathcal{S}$.  Note that $M\not =G'$, by Lemmas \ref{key1},  \ref{key3}, \ref{c2s3} and \ref{s3}, also by Lemma \ref{nilpotent}, $G/M$ is an extra-special group.

	Let $\overline G=G/N$, and let $\overline C=C_{\overline G'}(\overline M)\cong \overline M\times \overline L$. If $\overline L>1$, then $\overline C/\overline L\cong \overline M$ is a $p$-group for some $p\not=2$, a contradiction, as $|\overline M|<|\overline C|$. So,  $C_{\overline G}(\overline M)\cap \overline G'=\overline M$. Let $\overline T=\Pi_{\chi\in \Irr(\overline G|\overline M)}\ker(\chi)$. Then, as $\overline C=\overline M$, we conclude that $\overline T\cap \overline G'=1$. As $\overline{T}\cong \overline{T}\overline{M}/\overline{M}$ is a normal subgroup of $\overline{G}/\overline{M}$ whose intersection with $(\overline{G}/\overline{M})'$ is trivial, we get that $\overline T=1$.
	
	Then, $\chi$ is faithful for all $\chi \in  \Irr(\overline G|\overline M)$.
	Also $\cod(\chi)\not =2$, for all $\chi \in \Irr(\overline G|\overline M)$, by Lemma \ref{small}(c).
	So, $\overline G/\overline M$ acts on $\overline M$, Frobeniusly, by Lemma \ref{small2}(b) and $\overline G/\overline M$ must be $Q_{2^n}$, for some integer $n$. As $G/M$ is an extra-special $2$-group, we deduce that  $n=3$. Noting  that    nonlinear characters in  $\Irr(\overline G|\overline M)$ have distinct degrees ($m'_G(2)>2$ and $\cod(\chi)\not =2$ for all $\chi\in\Irr(\overline G|\overline M)$), we deduce  all nontrivial irreducible  characters of $\overline M$ are $\overline G$-conjugate and hence, $|\overline M|-1=|\overline G/\overline M|=8$ and so $\overline G$ is the described Frobenius group in the  statement of the lemma.

	%$$$$$$$$$$$$$$$$$$$$$$$$$$$$$$$$$$$$$$$$$$$$$$$$$$$$$$$$$$$$$$$$$$$
	Now, we prove that $N=1$. Suppose false, and assume $N$ is a minimal normal subgroup of $G$ of order $r^{a}$, for some prime $r$ and integer $a$. First, we show that $\ker(\chi)=1$, for all $\chi \in \Irr(G|N)$. Let $L=\Pi_{\chi\in \Irr(G|N)}\ker(\chi)$.  Then $LN\cong L\times N$ is a  normal subgroup of $G$.
	
	%%%%%%%%%%%%%%%%%%%%%%%%%%%%%%%%%%%%%%%%%%%%%%%%%%%%%%%%
	First assume, $L\cong LN/N\not =1$ and so $L\cong LN/N$ contains $M/N$. Note that  $M/N\cong C_3^2$. Then $M\cong N\times P$, where $P\cong C_3^2$,  is a normal subgroup of $G$. If $M/P$ is not a $2$-group, then replacing $(M,N)$ by $(M,P)$ in the above discussion we have $M/P\cong C_3^2$, which is not possible, since $G/P$ and $G/N$ have the same structure and so  they both are isomorphic to the Frobenius group $C_3^{2}\rtimes Q_8$, and  then $m'_G(9)>1$, which is contradicting the fact that $G$ is a $T_k'$-group. If $M/P$ is a $2$-group, then $N\leq {\bf Z}(G)$ and so $|N|=2$. Then $G/P$ must satisfy Lemma \ref{nilpotent}, which is not possible, as $|G/P|=16$.
	Therefore $L=1$ and $\ker(\chi)=1$, for all $\chi\in \Irr(G|N)$.   By Lemmas \ref{small}(c) and  \ref{small2}(b),   Sylow $p$-subgroups of $G/N$, for prime $p\not =r$, act on $N$, Frobeniusly. By the structure of  Sylow subgroups of $G/N$, we have $N$ is a $3$-group and Sylow $2$-subgroups of $G/N$, which are isomorphic to $Q_8$,   act  on $N$, Frobeniusly.  Then, $N\leq {\bf Z}(M)$  and since $M$ is the Sylow $3$-subgroup of $G$ using Lemma \ref{size Nq}, we get that $|N|=9$. Then $G$ is a Frobenius group of order $8\times 3^4$ and  $G$ has a unique minimal normal subgroup $N$. In addition,  characters in $\Irr(G|N)$ have distinct degrees, and by Lemmas \ref{small}(c) and  \ref{small2}(b),  $I_G(\lambda)/N\cong C_3^2$,  for all $\lambda \in \Irr(N)$. Also $\cd(G|N)=\{3\times 8\}$  using Lemma \ref{distinct}. Therefore, as $\cd(G/N)\cap\cd(G|N)=\emptyset$ and $G/N$ is a $D_0$-group (see \cite{D_0}), we conclude that   $G$ is a $D_0$-group,  contradicting the  main result of \cite{D_0}.
\end{proof}
\begin{lemma}\label{suz} 	Let $G$ be a  finite non-abelian   solvable $T'_k$-group for some integer $k\geq 1$,   such that $G/G'\cong C_4 $.  Then  $G\cong $Suz$(2)$ or Dic$_{12}=$SmallGroup(12,1).
\end{lemma}
\begin{proof}
	Assume $K$ is a normal subgroup of $G$ contained in  $G'$ such that $G'/K$ is a chief factor of $G$. Then,  write $\overline G\cong G/K$ and adopt the bar convention for this quotient group. First assume $\ker(\chi)\not =1$, for some $\chi\in \Irr(\overline G|\overline G')$, then   $\ker(\chi)\cong C_2$ or $C_4$. If $\ker(\chi)\cong C_4$, then $\overline G\cong \ker(\chi)\times \overline G'$ is abelian and we get a contradiction, hence $\ker(\chi)\cong C_2$. If $\overline G'$ is a $2$-group,  we get a contradiction, as  in this case $C_2\cong \overline G'\leq {\bf Z}(\overline G)$ and $\overline G/\overline G'\cong C_4$, which is not possible. 	
	So,  $\overline G'$ is an elementary abelian $p$-group, for some odd prime p, and by Lemma \ref{key1}, we get that $\overline G/\ker(\chi)\cong D_{2p}$. Therefore, $m'_{ G}(4)>1$ and $m'_{G}(p)=(p-1)/2$.  Hence,  $p=3$ and  $\overline G\cong$ Dic$_{12}$.

	Now, assume $\ker(\chi)=1$, for all $\chi\in \Irr(\overline G|\overline G')$.  Now, we claim that $\cod(\chi)\not= 4$, for all $\chi\in \Irr(\overline G| \overline G')$. Suppose false, then $|\overline G'|\times |I_G(\lambda)|/(\theta(1)|\overline G'|)=4$, for some $\lambda\in \Irr(\overline G')$ and $\theta\in \Irr(I_{\overline G}(\lambda)|\lambda)$, by Lemma \ref{small}(a). Therefore, $\overline G'$ is a $2$-group and hence $\overline G'\leq {\bf Z}(\overline G)$,  a contradiction. So, $\cod(\chi)\not =4$, for all $\chi\in\Irr(\overline G|\overline G')$. Therefore, members of  $\Irr(\overline G|\overline G')$ must have distinct degrees and hence $\overline G$ is a $D_0$-group. Using the classification of $D_0$-group, we get that $\overline G\cong $Suz$(2)$.

	Now, we  aim to prove that $K$ is trivial. On the contrary, we may assume $K$ is a minimal normal subgroup of $G$.
	
	First,  we show that $\ker(\chi)=1$ for all $\chi\in \Irr(G|K)$.
	Let $\overline G\cong $Dic$_{12}$,  assume $|\ker(\chi)|$ is divided by $2$, for some $\chi\in \Irr(G|K)$ and let $P\in \Syl_2(K)$. By the structure of Dic$_{12}$, we have $P\cong C_2$.
	Now,  $G/P$ satisfies Lemma \ref{s3}, and it follows that $G/P\in \{S_4, D_{18}, C_3^2\rtimes C_2\}$, and hence $m'_{G}(t)>1$, for some $t\not =4$, a contradiction. So $|\ker(\chi)|$ is not divided by $2$. Paying attention to  the structure of Dic$_{12}$ and Suz$(2)$  we have 
	$G'\cong K\times T $, where $T=\ker(\chi)\cap G'\cong C_5$ or $C_3$ and applying the above discussion on $G/T$, we get that $G/K\cong $Suz$(2)$ or Dic$_{12}$,  and so $T\cong C_5$ or $C_3$.  If $G/T\cong G/K$, then $m'_G(t)>1$ for some $t\in \{6,5\}$(both $G/K$ and $G/T$ are isomorphic to Suz$(2)$ or Dic$_{12}$) and, as $m'_G(4)>1$, we get  a contradiction. Therefore one of those groups is isomorphic to Dic$_{12}$ and the other one is isomorphic to Suz$(2)$. So $G'\cong C_{15}$, has $8$  irreducible characters of order 15, whose inertia subgroups are $G'$. Therefore, $G$ has $2$ distinct irreducible characters of the same degree $4$. It is easy to see that both of those characters are faithful and so $m'_G(15)>1$,  a contradiction.   So $\ker(\chi)=1$, for all $\chi\in \Irr(G|K)$.
	
	We claim that $\cod(\chi)\not =4$, for all $\chi\in \Irr(G|K)$. Otherwise, $|K|\times |I_G(\lambda)|/(\theta(1)|K|)=4$, for some $\lambda\in \Irr(K)$ and $\theta\in \Irr(I_G(\lambda)|\lambda)$, by Lemma \ref{small}(a). Since all Sylow subgroups of $G/K$ are cyclic, then $\lambda$ extends to $I_G(\lambda)$ (see \cite[Corollary 11.31]{isaac}), and so $\theta(1)$ is a $2$-element (by noticing the character degrees  of Dic$_{12}$ and Suz$(2)$). Hence,  both $|K|$ and $|I_G(\lambda)|/|K|$ are $2$-groups, whose orders are at most 4 and then,  $3\leq |G:I_G(\lambda)|$,  while $|K|\leq 4$. By Remark \ref{orbit}, the only possibility is that $|K|=4$ and $|G:I_G(\lambda)|=3$, which means that $G/K\cong$Dic$_{12}$,  $I_G(\lambda)/K\cong C_4$ and $\theta(1)=4$, a contradiction.  So our claim is proved.
	
	Therefore,  the members of $\Irr(G|K)$ has distinct degrees (as $m'_G(4)>1$). Since all Sylow subgroups of  $I_G(\lambda)/K$ are cyclic, then $\lambda$ is extendible to $I_G(\lambda)$ (see \cite[Corollary 11.31]{isaac}),  and using Gallagher's theorem, we deduce that $I_G(\lambda)/K=1$, for all $\lambda \in \Irr(K)$.  It follows that $G/K$ acts on $K$, Frobeniusly, but it is  a contradiction, when $G/K\cong $Suz$(2)$. Then we may assume  $G/K\cong$Dic$_{12}$. Then,  $\cd(G|K)=\{12\}$ and so  $\cd(G|K)\cap\cd(G/K)=\emptyset$. On the other hand $G/K$ is a $D_1$-group, which implies that $G$ is a $D_1$-group. Now, looking at the structure of $D_1$-groups  (see \cite{d1}), we get a contradiction.  	
\end{proof}
{\bf Proof of the Main Theorem case (b):} Easy calculation shows us that groups in cases (1) and (2) are $T'_k$-groups, for some $k\leq 3$. Also by Lemma \ref{abeliana}, elementary abelian group $G$ is a  $T_{|G|-1}'$-group. If $G$ is an extra-special $2$-group, since it has a unique nonlinear irreducible character and all of nontrivial linear characters have the same codegrees, we get that $G$ is a $T_{|G/G'|-1}'$-group. Therefore the described groups in the case (3) are  $T'_{|G/G'|-1}$-groups.  

Let  $G$ be a Frobenius group isomorphic to $C_2^{\beta}\rtimes C_{2^{\beta}-1}$, where $\beta$ is an integer and $p=2^{\beta}-1$ is  prime. Then, $G$ has a unique nonlinear irreducible character and all of its nontrivial linear irreducible characters have the same codegree $p$. Therefore $G$ is a $T_{(p-1)}'$-group.

Let $G$ be a Frobenius group isomorphic to $C_p^{\beta}\rtimes C_2$, for some prime $p$ and  integer $\beta$. We know that $\lambda^G\in \Irr(G)$ for every $\lambda\in \Irr(G')$ and $|\ker(\lambda^G)|=p^{\beta-1}$. Since the size of  $G$-orbit of each nontrivial element in $\Irr(G')$ is $2$, then $G$ has exactly $(p^{\beta}-1)/2$ nonlinear irreducible characters with the same codegree $p$. 	 As $G$ has just one linear nontrivial irreducible character, we deduce that $G$ is a $T_{(p^{\beta}-1)/2}'$-group.

Now,  we just need to prove the "only if" part of the theorem.  If $G$ is abelian, by Lemma \ref{abeliana}, we have either $G$ is an elementary abelian $p$-group, for some prime $p$,  or  $G\cong C_4$. So we assume $G$ is not abelian. As $G/G'$ is a $T_l'$-group, for some integer $l$, we deduce that $G/G'$ is a cyclic group of order $4$ or an elementary abelian group.  First, assume $G/G'\cong C_p^{\alpha}$, for some prime $p$ and integer $\alpha$. Let $K$ be a normal subgroup of $G$, contained in $G'$, such that  $G'/K$ is a chief factor of $G$. If $G'/K$ is not a $p$-group, then by Lemma \ref{key1}, we have $G/K$ is one of the groups described in the statement of Lemma \ref{key1}. First, let $G/K$ be  isomorphic to a Frobenius group of order $p2^{\beta}$, where  $p=2^{\beta}-1$. Then applying Lemma \ref{key3}, we have $K=1$ and $G$ is  the group described in case (4) of the main theorem. If $G/K\cong  C_2\times S_3$, applying Lemma \ref{c2s3}, $K=1$ and $G$ is one of the groups, described in case (2) of the main theorem. If $G/K\cong D_{2q}$, for some odd prime $q$, then applying Lemma \ref{s3}, either $G\cong$$S_4$, $D_{18}$ or it is isomorphic to the Frobenius group described in case (5) of the main theorem. Now, let $G'/K$ be a $p$-group. Then, applying Lemma \ref{lastone}, $G$ is one of the groups described in cases (2) and (3). At last, suppose $G/G'\cong C_4$. Hence, applying Lemma \ref{suz}, $G\cong $Suz$(2)$ or Dic$_{12}$, which are two of the groups in case (1) of the main theorem. $\Box$

%\affiliationthree{~} %inserts a space to make this field empty
%\affiliationfour{%
%	Current address:\\
%	Present long-term address\\
%	Country
%	\email{t.hird@institution.edu}}
%\email{m.ebrahimi.math@ipm.ir}
%\email{z$\_$akhlaghi@aut.ac.ir}

%\address{$^{3}$ Department of Pure Mathematics, Faculty of Mathematics and Statistics, University of Isfahan, Isfahan, 81746-73441, Iran.}
%\email{m.khatami@sci.ui.ac.ir}
% ------------------------------------------------------------------------

\begin{thebibliography}{1}
	%\bibitem{acdel}
	%Akhlaghi Z.,   Casolo, C.,  Dolfi, S.,  Pacifici E., Sanus, L. (2019) On the character degree graph of finite
	%groups. {\it Annali di Mat. Pura Appl.} DOI: 10.1007/s10231-019-00833-0.
	
	
	%\bibitem{codegree}
	%Alizadeh, F., Behravesh, H.,   Ghaffarzadeh, M.,  Ghasemi,  M.,  and  Hekmatara, S.  (2019) \newblock Groups with few codegrees of irreducible characters.
	%\newblock {\em Comm. Algebra}. 47: 1147--1152.
	
	%Alizadeh, F., Behravesh, H., Ghaffarzadeh, M., Ghasemi, M. and Hekmatara, S., 2019. Groups with few codegrees of irreducible characters. Communications in Algebra, 47(3), pp.1147-1152.
	
	%\bibitem{carlo}
	%Casolo,C. (2010).
	%\newblock Some linear actions of finite groups with $ q^{\prime} $-orbits,
	%\newblock {\em J. Group Theory}.
	%13: 503--534.
	
	\bibitem{d1}
	{Y. Berkovich},	
	`Finite solvable groups in which only two nonlinear irreducible characters have equal degrees',
	{\em J. Algebra},
	184 (1996), 584--603.
	
	\bibitem{D_0}
	{Y. Berkovich, D. Chillag \and M. Herzog},
	`Finite groups in which character  degrees of nonlinear irreducible characters are distinct',
	{\em Proc. Amer. Math. Soc.},
	115 (4) (1992), 955--959.
	
	\bibitem{4 D2}
	{ Y. Berkovich, I. M. Isaacs \and  L. Kazarin},
	`Groups with distinct monolithic character degrees',
	{\em J. Algebra},
	216 (1) (1999),
	%no. 1,
	448--480.
	
	
	
	
	\bibitem{berkovich}
	{Y. Berkovich \and L. Kazarin},
	`Finite nonsolvable groups in which only two nonlinear irreducible characters have equal degrees',
	{\em J. Algebra},
	184 (1996), 538--560.
	
	%\bibitem{ble} M. Bianchi, D. Chillag, M. L. Lewis, E. Pacifici, {\}
	%\bibitem{carlo}
	%C. Casolo,
	%\newblock Some linear actions of finite groups with $ q^{\prime} $-orbits,
	%\newblock {\em J. Group Theory}
	%\textbf{13} (2010), no. 4, 503--534.
	
	\bibitem{ble}
	{ M. Bianchi, D. Chillag, M. L. Lewis \and E. Pacifici},
	`Character degree graphs that are complete graphs',
	{\em Proc. Amer. Math. Soc.},
	135(3) (2007), 671--676.
	
	
	%\bibitem{7}
	%R. Carter,
	%\newblock Finite Groups of Lie Type,  Conjugacy classes and complex characters,
	%John Wiley and Sons, New York, (1985).
	
	%\bibitem{atlas}
	%J. H. Conway, R. T. Curtis,  S. P. Norton and R. A.  Wilson,
	%\newblock  Atlas of finite groups,
	%Oxford: Oxford University Press, (1985).
	\bibitem{dim} { M. R. Darafsheh, A. Iranmanesh \and S. A. Moosavi},
	`Groups whose nonlinear characters are rational valued',
	{\em Arch. Math.},
	94 (2010), 411–418.
	
	\bibitem{DL}
	{N. Du \and M. L. Lewis},
	`Codegrees and nilpotence class of $p$-groups',
	{\em J. Group Theory},
	19 (2015), 561--567.
	
	%\bibitem{lewis}
	%Croome, S.,     Lewis, M. L. (2019)
	%\newblock $ p $-groups with exactly four codegrees. Preprint https://arxiv.org/abs/1901.07425.
	
	
	
	%\bibitem{chillag}
	%Chillag, D.,  Mann   A.,  Manz, O. (1991).
	%\newblock The co-degrees of irreducible characters.
	%\newblock {\em Israel J. Math.}
	%73: 207--223.
	
	
	
	%\bibitem{new}
	%Du N., Lewis, M. L. (2016)
	%\newblock Codegrees and nilpotence class of $ p $-groups,
	%\newblock {\em J. Group Theory}.
	%19:    561--567.
	
	
	
	
	
	
	%\bibitem{regular}
	%Goodwin, D.P.  (2000).
	%\newblock Regular orbits of linear groups with an application to the $ k(GV) $-problem.
	%\newblock {\em J. Algebra}.
	%227: 395--432.
	
	\bibitem{ebrahimi}
		M. Ebrahimi,
	`Groups in which the co-degrees of the irreducible characters are distinct', Preprint (2020), https://arxiv.org/abs/2008.02433.
	
	
	%\bibitem{Isaacs}
	%I. M. Isaacs,
	%\newblock Element orders and character codegrees,
	%\newblock {\em Arch. Math. (Basel)},
	%\textbf{97} (2011), 499--501.
	
	\bibitem{isaac}
	{I. M. Isaacs},
	`Character theory of finite groups',
	{\em New York, Academic Press}, (1976).
	
	\bibitem{isaacs}
	{	I. M. Isaacs},
	`Finite group theory',
	{\em Amer. Math. Soc, Providence}, (2008).
	
	%\bibitem{d2}
	% Liu, Y.,    Lu,  Z.Q. (2015).
	%%\newblock {\em Acta Math. Sin. Engl. Ser.}
	%31: 1683--1702.
	
	%\bibitem{kerber}
	%G.  James and A. Kerber,
	%\newblock The representation theory of the symmetric group,
	%\newblock  Encyclopedia Math. Appl., vol. 16, Addison–Wesley,
	%(1981).
	
	\bibitem{group}
	{D. Liang \and G. Qian},
	`Finite groups with coprime character degrees and codegrees',
	{\em J. Group Theory},
	19 (2016), 763--776.
	
	
	%%%\newblock Finite nonsolvable groups with many distinct character degrees.
	%\newblock {\em Pacific J. Math.}
	%268 (2): 477--492.
	
	\bibitem{K}
	{Y. Liu \and Z. Q. Lu},
	`Solvable $D_2$-groups',
	{\em Acta Math. Sin. (Engl. Ser.)},
	33 (2017), 77--95.
	
	%DOI: 10.1007/s10114-016-5353-2.
	
	
	%\bibitem{8}
	%J. H. Conway, R. T. Curtis, S. P. Norton, R. A. Parker, R. A. Wilson,
	%\newblock {\em  Atlas of finite groups},
	%\newblock (Oxford University Press, Eynsham, 1985).
	
	
	
	
	\bibitem{d2}
	{Y. Liu  \and Z. Q. Lu},
	`Non-solvable $ D_{2} $-groups',
	{\em Acta  Math. Sin. (Engl. Ser.)},
	31(11) (2015),
	%no. 11,
	1683--1702.
	
	%\bibitem{malle}
	%G. Malle and A. Moret\'o,
	%\newblock   Nonsolvable groups with few character degrees,
	%\newblock {\em J.  Algebra},
	%\textbf{294} (2005), 117--126.
	
	\bibitem{Qian}
	{G. Qian,  Y.  Wang \and H. Wei},
	`Codegrees of irreducible characters in finite groups',
	{\em J.  Algebra},
	312 (2007), 946--955.
	
	
	%\bibitem{Schmid}
	%P. Schmid,
	%\newblock Extending the Steinberg representation,
	%\newblock {\em J. Algebra},
	%\textbf{150} (1992), 254--256.
	
	\bibitem{zzk} { Z. Sayanjali, Z. Akhlaghi \and B. Khosravi}, `On the multiplicity of character degrees of non-solvable groups',
	{\em J. Algebra Appl.}, DOI: 10.1142/S0219498821500766.
	%\bibitem {gap}
	%The GAP Group, GAP–-Groups, Algorithms, and Programming, Version 4.8.8; 2017. (http://www.gap-system)
	%org.
	
	
	\bibitem{hung}
	{H. P. Tong-Viet},
	`Finite nonsolvable groups with many distinct character degrees',
	{\em Pacific J. Math.},
	268 (2) (2014), 477--492.
	
	
	
\end{thebibliography}
\end{document}